\documentclass[12pt]{amsart}

\usepackage[matrix,arrow,curve,frame]{xy}
\usepackage{amsmath,amsthm,amssymb,enumerate}
\usepackage{latexsym}
\usepackage{amscd}
\usepackage[colorlinks=true]{hyperref}
\usepackage{euscript}

\newtheorem{thm}[subsection]{Theorem}
\newtheorem{defn}[subsection]{Definition}

\newtheorem{claim}[subsection]{Claim}
\newtheorem{corr}[subsection]{Corollary}

\newtheorem{remark}{Remark}

\theoremstyle{definition}
\newtheorem{example}[subsection]{Example}

\newcommand{\R}{\mathbb R}

\newcommand{\Z}{\mathbb Z}

\newcommand{\C}{\mathbb C}

\newcommand{\T}{\mathbb T}

\newcommand{\lt}{{\mathcal{T}}}

\DeclareMathOperator{\hocolim}{hocolim}

\DeclareMathOperator{\Map}{Map}

\DeclareMathOperator{\Rep}{Rep}

\DeclareMathOperator{\Tr}{T}
\DeclareMathOperator{\W}{W}
\DeclareMathOperator{\K}{K}
\DeclareMathOperator{\Hg}{H}

\DeclareMathOperator{\Hbg}{\tilde{H}}
\DeclareMathOperator{\FT}{N({\bf T})}

\DeclareMathOperator{\FTI}{N_I({\bf T})}

\DeclareMathOperator{\BU}{BU}

\DeclareMathOperator{\TMF}{TMF}
\DeclareMathOperator{\LM}{LM}
\DeclareMathOperator{\LG}{LG}
\DeclareMathOperator{\LGT}{L(G/T)}
\DeclareMathOperator{\LT}{LT}
\DeclareMathOperator{\LbT}{\tilde{L}T}
\DeclareMathOperator{\LbG}{\tilde{L}G}
\DeclareMathOperator{\M}{M}
\DeclareMathOperator{\SU}{SU}
\DeclareMathOperator{\GT}{G/T}
\DeclareMathOperator{\Lgt}{L(G/T)}

\DeclareMathOperator{\GHgI}{\Gr/\Hg_I}

\DeclareMathOperator{\MT}{M/T}

\DeclareMathOperator{\Gr}{G}

\DeclareMathOperator{\No}{N}

\DeclareMathOperator{\Lo}{L}

\newfont{\german}{eufm10}

 \DeclareMathOperator*{\invlim}{\varprojlim}

\newcommand\qu{/\kern-.7ex/}

\begin{document}
\pagestyle{plain}

\title
{Quantization of the Modular functor and Equivariant Elliptic cohomology}
\author{Nitu Kitchloo}
\address{Department of Mathematics, Johns Hopkins University, Baltimore, USA}
\email{nitu@math.jhu.edu}
\thanks{Nitu Kitchloo is supported in part by the Simons Fellowship.}

\date{\today}


{\abstract

\noindent
Given a simple, simply connected compact Lie group $\Gr$, let $\M$ be a $\Gr$-space. We describe the quantization of the category of positive energy representations of the loop group of $\Gr$ at a given level and parametrized over $\LM$. This procedure is described in terms of dominant $\K$-theory of the loop group evaluated on the phase space given by the tangent bundle of basic gauge fields about a circle (parametrized over $\LM$ and with gauge symmetries $\LG$). As such, our construction gives rise to a {\em categorical} BV-BRST type quantization for families of rational 2d CFTs with gauge symmetries parametrized over $\M$. More concretely, we construct a holomorphic sheaf over a universal elliptic curve with values in dominant $\K$-theory of the loop space $\LM$, and show that each stalk of this sheaf is a cohomological functor of $\M$. We also interpret this theory as a model of equivariant elliptic cohomology of $\M$ as constructed by Grojnowski and others.}
\maketitle

\tableofcontents

\section{Introduction:}

\medskip
\noindent
Given a string manifold $\M$, Witten has used heuristic arguments to show that the genus one partition function of a certain 2-dimensional sigma model (with background $\M$) can be seen as the value of a genus evaluated on $\M$. In particular, this ``Witten genus of $\M$" takes values in modular forms (see \cite{S3} for an overview). This led to a flurry of activity aimed at constructing the underlying ``Elliptic cohomology theory" with coefficients being modular forms, and which is a receptacle for the Witten genus. Subsequent work by M. Hopkins and his collaborators resulted in the construction of the theory $``\TMF"$ (Topological Modular Forms). This theory has been shown to admit all the homotopical properties one would expect of an ``Elliptic Cohomology theory" (indeed, it is a universal elliptic cohomology theory in a suitable sense). However, a geometric description of $\TMF$ that allows one to draw a connection to physics remains elusive. 

\medskip
\noindent
Given the chiral sector of a two dimensional conformal field theory, one expects the genus one partition function to be the character of a representation of an underlying ``Chiral Algebra" (or Vertex Algebra). The category of representations of the chiral algebra is typically the linear category that one assigns to a circle in the process of constructing the underlying ``Modular Functor" \cite{S}. If the field theory in question is reasonably nice (i.e. rational), then this category of representations is semi-simple. This suggests, in particular, that the $\K$-theory of this category is essentially the same as topological complex $\K$-theory (possibly twisted by an anomaly).

\medskip
\noindent
Now consider two dimensional field theories parametrized over a manifold $\M$. The basic fields in dimension one are parametrized over the space $\LM$ of smooth maps from a circle to $\M$. The space $\LM$ comes with a manifest action of the rotation group $\T$. Motivated by the observation made in the previous paragraph, in \cite{KM} J. Morava and the author considered the completion of the $\T$-equivariant $\K$-theory of $\M$ at the rotation character $q$: $\K_{\T}(\LM)((q))$. One may interpret this completion as localizing around infinitesimal loops (or the low-energy limit). A simple argument was used in \cite{KM} to show that this is a cohomology theory in $\M$ (i.e. satisfies the Mayer-Vietoris axioms on $\M$) and can be interpreted as an approximation to elliptic cohomology of $\M$ at the ``Tate locus". In particular, $\K_{\T}(\LM)((q))$ was shown to be a receptacle for the Witten genus.   

\medskip
\noindent
Equivariantly, the situation is much more interesting and gives further support to the idea of relating the $\K$-theory of the chiral algebra to the topological $\K$-theory of the space of fields on a circle. Consider two dimensional field theories with local gauge symmetries for an underlying simple and simply connected Lie group $\Gr$. The space of fields over a circle for such theories supports an action of the loop group $\LG$. In addition, the basic gauge fields for such theories is the $\LG$-space $\mathcal{A}$ of principal connections on the trivial $\Gr$-bundle over a circle. Since the quantum state-space of a two dimensional field theory has a discrete positive (or negative) energy spectrum, one expects the state-space to represent an element in equivariant $\K$-theory $\K_{\LG}(\mathcal{A})$ - if the latter can be rigorously defined. Freed-Hopkins-Teleman have shown \cite{FHT} that once we incorcorporate a twisting on $\K_{\LG}(\mathcal{A})$ induced by the level (i.e. the central character of the universal central extension $\LbG$ of $\LG$), this $\K$-theory is indeed well defined, and becomes canonically isomorphic to Grothendieck group of positive energy representations of $\LG$.

\medskip
\noindent
In this document, we would like to offer the suggestion that studying parametrized 2d conformal field theories that admit gauge symmetries via the $\K$-theory (of the category of representations of the underlying Chiral algebra), is a richer object of study. In particular, one would like to evaluate this $\K$-theory on the stack of {\em classical} fields in codimension one (or fields that satisfy equations of motion in a neighborhood of a codimension one manifold), and interpret it as a ``BV-BRST type quantization" of the modular functor underlying these field theories. More discussion on this philosophy can be found at \cite{NC}. Let us justify this general idea with details in a concrete setting that is relevant to this article. 

\medskip
\noindent
Let us begin by assuming that one may define the $\LG$-equivariant $\K$-theory for spaces with a proper action of the group $\T \ltimes \LbG$, where $\LbG$ denotes the universal central extension of the loop group with a compatible lift of the rotation group $\T$. We apply this $\K$-theory to the ``phase space" for 2d conformal field theories with gauge symmetry, parametrized over a $\Gr$-space $\M$. This space is defined as the stack $(\LM \times \mathcal{A}_\C)$, where $\mathcal{A}_\C$ denotes the complexification of the stack $\mathcal{A}\qu (\T \ltimes \LbG)$ and should be understood as the tangent bundle of the fields on a circle\footnote{The tangent vectors are known as conjugate-momenta}. As such, the phase space encodes the {\em classical solutions} to any second order equations of motion on the space of connections on the germ of a cylinder, supporting the action of the Gauge group (see remark \ref{complexification}). Extrapolating from the case of the trivial group $\Gr$ considered in \cite{KM}, we can only expect to get an equivariant cohomology theory in $\M$ from this data once we suitably complete with respect to the character $q$ that represents energy. In this document, we describe how one may successfully do so for simple and simply-connected compact Lie groups $\Gr$ and interpret it as an approximation of elliptic cohomology of the parameter space\footnote{Much of this program should carry through with minor alterations for arbitrary compact connected Lie groups}.  

\medskip
\noindent
To begin the program as described above, one must first construct a $\T \ltimes \LbG$-equivariant version of $\K$-theory for any positive integral level. One can interpret this as a ``{\em Categorical Gauging}" mechanism. This is precisely what has been constructed by the author in \cite{Ki}, and goes by the name of Dominant K-theory which is reviewed in Section \ref{back}. The next step (which is the heart of this article) is described in Section \ref{locality} and can be interpreted as a ``{\em Categorical BV-BRST type quantization}" procedure, which constructs, for any positive level $k$, a global version of dominant $ \K$-theory: $ {}^k \mathcal{K}_{\T \ltimes \LbG}(\LM)$. In other words, we describe ${}^k \mathcal{K}_{\T \ltimes \LbG}(\LM)$ as an equivariant holomorphic sheaf over the stack $\{ (i \R_+ \times \mathcal{A} )\qu \T \ltimes \LbG \}_\C$ built out of dominant $ \K$-groups of the space $\LM$ (see  theorem \ref{dom}). Here $i \R_+$ denotes the positive energy axis, and its complexification can be identified with the upper half plane $\mathfrak{h}$ which admits an interpretation as a moduli space of parametrized annuli (see remark \ref{complexification}). As such, the phase space $\{ (i \R_+ \times \mathcal{A} )\qu \T \ltimes \LbG \}_\C$ may also be interpreted as a moduli space of $\Gr$-bundles parametrized over the moduli space $\mathfrak{h}$ of annuli. In this context, our sheaf ${}^k \mathcal{K}_{\T \ltimes \LbG}(\LM)$ has a description in terms of the modular functor (in genus-zero and with two insertions) for loop group representations of a fixed level (see remark \ref{cb}). If $\M$ is a finite $\Gr$-space (i.e. $\M$ is equivalent to a finite $\Gr$-CW complex), we show that the stalks of this sheaf are cohomological functors of $\M$. In Section \ref{grojnowski}, we take invariants with respect to certain gauge subgroups, so that ${}^k \mathcal{K}_{\T \ltimes \LbG}(\LM)$ descends to a sheaf ${}^k \mathcal{G}(\M)$ on a universal elliptic curve. 
By taking $\M$ to be a point and evaluating invariant global sections of the above sheaf, we identify the coefficients of this theory with Weyl invariant theta functions, or equivalently, with level $k$ positive energy representations of the loop group $\T \ltimes \LbG$ (see corollary \ref{rep LG} and remark \ref{geometric} ). In Section \ref{LGT} we take $\M$ to be the full flag variety $\GT$ and recover all theta functions, or equivalently, the level $k$ representations of $\T \ltimes \LbT$ (see theorem \ref{rep LT}). The question of modulariy is taken up in Section \ref{modularity}. In particular, for $\Gr$-spaces $\M$ that satisfy some well-known conditions, the space of invariant sections of these sheaves is shown to be a representation of the modular group $\mbox{SL}_2(\Z)$ (see theorem \ref{modular}). 

\medskip
\noindent
As mentioned earlier, our construction resembles the $\Gr$-equivariant elliptic chomology of $\M$ as constructed by Grojnowski \cite{G}, and subsequently explored in more detail by Ando and others (see \cite{AB}). These constructions identify equivariant elliptic cohomology as a twisted sheaf of algebras over the universal elliptic curve\footnote{Lurie also has an algebraic interpretation of this sheaf (see Sections 3 and 5 in \cite{L})}. We expect ${}^k \mathcal{G}(\M)$ to be closely related (if not isomorphic) to Grojnowski's equivariant elliptic cohomology (see remark \ref{main1} and Section \ref{remarks}).

\medskip
\noindent
The author wishes to thank the American Institute of Math., and all the participants of the AIM (SQUARE) workshop in mathematical physics (in particular M. Ando, H. Sati and J. Morava) for inspiring the ideas that led to this paper. We would also like to thank Owen Gwilliam for helpful discussions on the BV-BRST formalism. 

\medskip
\noindent
One word about our conventions: Throughout this article, we deal with several actions (both left and right) as well as extensions of groups. In order to avoid confusion, we will use the notation $g \ast x$ to mean that a group element $g$ acts on an element $x$. The notation $g \, h$ is reserved for a product of two group elements $g$ and $h$ inside a larger group. 

\section{Background on Dominant $ \K$-theory and the space $\mathcal{A}$:}  \label{back}

\medskip
\noindent
Let $\Gr$ be a simple and simply-connected compact Lie group of rank $n$. Let $\LG$ denote 
the loop group of $\Gr$. The group $\LG$ supports a universal central extension \cite{PS} which will henceforth be denoted by $\LbG$. The action of the rotation group $\T$ on $\LG$ lifts to an action on $\LbG$, so that one may define the extended loop group: $\T \ltimes \LbG$. 

\medskip
\noindent
Given a representation of $\T \ltimes \LbG$ in a separable Hilbert space $\mathcal{H}$, we say it has level $k$, if the central circle $S^1$ acts by the character $e^{ik\theta}$. It is well known that for positive $k$, the category of representations of level $k$ is semi simple, with finitely many irreducible objects. In addition, one may prescribe an orientation on the circle $\T$ so that any irreducible representation of level $k >0 $, has finitely may negative Fourier modes (i.e negative powers of the character $q$). Such representations are called positive energy representations. 

\medskip
\noindent
As explained in \cite{Ki}, we fix a level $k>0$ and consider a level $k$ representation of $\T \ltimes \LbG$ in a separable Hilbert space $\mathcal{H}_k$ with the property that any level $k$ irreducible representation occurs infinitely often in $\mathcal{H}_k$. Let $\mathcal{F}_k$ denote the space of Fredholm operators on $\mathcal{H}_k$. One may choose a topology on $\mathcal{F}_k$ so that its underlying homotopy type is $\Z \times \BU$ and admits a continuous action of $\T \ltimes \LbG$. In \cite{Ki} we constructed a two periodic cohomology theory, called Dominant $\K$-theory on the category of proper $(\T \ltimes \LG)$-CW complexes:\footnote{Note that the center acts trivially on these spaces.} 
\[ {}^k\K^{0}_{\T \ltimes \LbG}(X) := \pi_0 \Map^{\T \ltimes \LbG} \{ X, \mathcal{F}_k \}, \quad \quad {}^k \K^{-1}_{\T \ltimes \LbG}(X) := \pi_1 \Map^{\T \ltimes \LbG} \{ X, \mathcal{F}_k \}, \]
where $\Map^{\T \ltimes \LbG} \{ X, \mathcal{F}_k \}$ denotes the space of continuous $(\T \ltimes \LbG)$-equivariant maps from $X$ to $\mathcal{F}_k$. 

\medskip
\begin{defn}
Having defined Dominant $\K$-theory, let us set some notation. Given a closed subgroup $\Hg \subset \T \ltimes \LG$, let $\Hbg$ be the induced central extension. Given a proper $\Hg$-space $Y$, we define ${}^k \K^{-\ast}_{\Hbg}(Y)$ by:
\[ {}^k \K^{-\ast}_{\Hbg}(Y) := {}^k \K^{-\ast}_{\T \times \LbG}((\T \ltimes \LbG)_+ \wedge_{\Hbg} Y) = \pi_\ast \Map^{\T \ltimes \LbG} \{ (\T \ltimes \LbG)_+ \wedge_{\Hbg} Y, \mathcal{F}_k \} = \pi_\ast \Map^{\Hbg} \{ Y, \mathcal{F}_k \}. \]
\end{defn}

\medskip
\noindent
We now describe the structure of the space $\T \ltimes \LbG$-space $\mathcal{A}$ of principal connections on the trivial $\Gr$-bundle over the circle. Indeed, this space is homeomorphic to a proper, finite $\T \ltimes \LbG$-complex. In addition, it is the universal space for proper actions, in that any other proper $\T \ltimes \LbG$-space maps to it along an equivariant map that is unique up to an equivariant homotopy.

\medskip
\noindent  
Fix a maximal torus $\Tr$ of $\Gr$, and let $\alpha_i$, $1 \leq i \leq n$ be a 
set of simple roots. We let $\alpha_0$ denote the highest root. Each root 
$\alpha_i$, $0 \leq i \leq n$ determines a compact subgroup $\Gr_i$ of $\Gr$. 
More explicitly, $\Gr_i$ is the semi simple factor in the centralizer of the 
codimension one subtorus given by the kernel of $\alpha_i$. Each $\Gr_i$ may 
be canonically identified with $\SU(2)$ via an injective map $\varphi_i : \SU(2) \longrightarrow \Gr$. 
We use these groups $\Gr_i$ to define corresponding compact subgroups $\Gr_i$ 
of $\LG$ as follows:
\[ 
\Gr_i = \{ z \mapsto \varphi_i\left(\begin{matrix} a & b \\ c & d 
\end{matrix}\right)  \quad \quad \quad \mbox{if} \quad  \left( 
\begin{matrix} a & b \\ c & d \end{matrix}\right) \in \SU(2) \quad i>0,  \} \]
\[ \Gr_0 = \{ z \mapsto \varphi_0\left(\begin{matrix} a & bz \\ cz^{-1} & d 
\end{matrix}\right) \quad \mbox{if} \quad  \left( \begin{matrix} a & b \\ 
c & d \end{matrix}\right) \in \SU(2).\} \quad \quad \quad 
\]

\noindent
Note that each $\Gr_i$ is a compact subgroup of $\LG$ isomorphic to 
$\SU(2)$. Moreover, $\Gr_i$ belongs to the subgroup $\Gr$ of constant loops if 
$i \geq 1$. The rotation group $\T$ preserves each $\Gr_i$, acting trivially 
on $\Gr_i$ for $i \geq 1$, and nontrivially on $\Gr_0$.

\medskip
\begin{defn} 
For any proper subset $I \subset \{0,1,\ldots,n\}$, let $\Gr_I$ denote the group generated by $\Gr_i$, $i \in I$, and let the parabolic subgroup $\Hg_I \subset \LG$ be the group generated by $\Tr$ and $\Gr_I$. For the empty set, we define $\Hg_I$ to be 
$\Tr$. Similarly, we define $\Hbg_I \subset \LbG$ to be the induced central extension of $\Hg_I$. 
\end{defn}

\medskip
\begin{remark}
Henceforth, we only consider proper subsets $I \subset \{0,1,\ldots,n\}$. The groups $\Hbg_I$ are compact Lie groups that are preserved under the action of the rotation group $\T$, with $\T$ acting nontrivially on $\Hbg_I$ if and only if $0 \in I$. In particular, $\T \ltimes \Hbg_I \subseteq \T \ltimes \LbG$ is a well defined compact Lie subgroup. In addition, all representations of $\T \ltimes \Hbg_I$ of level $k$ appear in $\mathcal{H}_k$ for $k>0$. It follows therefore that ${}^k \K_{\T \ltimes \Hbg_I}(X)$ is a module over standard equivariant $\K$-theory: $\K_{\T \ltimes \Hg_I}(X)$ for any $\T \ltimes \Hbg_I$-space $X$. \end{remark}

\noindent
\medskip
The following theorem was proven in \cite{KM}:

\medskip
\begin{thm}
The space $\mathcal{A}$ of principal $\Gr$ connections on the trivial bundle $\Gr \times S^1$ is $\T \ltimes \LbG$-equivariantly homeomorphic to the ``Topological Tits building" \cite{Ki} given by the the homotopy colimit of homogeneous spaces over all proper subsets $I \subset \{0,1,\ldots,n\}$:
\[ \hocolim_{ I \subset \{0,1,\ldots,n\}} ( \T \ltimes \LG )/ ( \T \ltimes \Hg_I ) =  \hocolim_{ I \subset \{0,1,\ldots,n\}} ( \T \ltimes \LbG )/ ( \T \ltimes \Hbg_I ). \]
Furthermore, the space $\mathcal{A}$ is the universal space for proper actions of $\T \ltimes \LbG$.
\end{thm}

\medskip
\noindent
Consider the inclusion of the maximal torus $ {\bf{T}} := \T \times \Tr \times \, S^1 \subset \T \ltimes \LbG$. Let $\FT$ be the normalizer of the maximal torus. Recall that $\FT$ is an extension of a discrete group $\tilde{\W}(\Gr)$ by the torus ${\bf T}$: 
\[ 1 \rightarrow {\bf T} \longrightarrow \FT \longrightarrow \tilde{\W}(\Gr) \rightarrow 1. \]
The group $\tilde{\W}(\Gr)$ is the Weyl group of $\T \ltimes \LG$ known as the Affine Weyl group. It  is equivalent to $\pi_1(\Tr) \rtimes \W(\Gr)$, where $\W(\Gr)$ is the Weyl group of $\Gr$.

\medskip
\noindent
The fixed subspace of $\mathcal{A}$ under ${\bf{T}}$ is given by the universal space for proper $\FT$-actions: 
\[ \Sigma = \mathcal{A}^{{\bf{T}}} = \hocolim_{I \subset \{0,1,\ldots,n\}} \FT/\FTI =  \hocolim_{I \subset \{0,1,\ldots,n\}} \tilde{\W}(\Gr)/\W_I, \]
where $\FTI$ is the normalizer of ${\bf{T}}$ in $\Hbg_I$, and $\W_I$ is the corresponding Weyl group. The space $\Sigma$ is in fact the space of principal connections on the trivial principal $\Tr$-bundle over a circle. In particular, it is homeomorphic to the Lie algebra $\lt$ of $\Tr$. 

\medskip
\begin{remark} \label{affine}
Consider the subspace $\Delta \subset \Sigma$:
\[ \Delta := \hocolim_{I \subset \{0,1,\ldots,n\}} ({\bf T}/{\bf T}) \, \, \subseteq \, \, \hocolim_{I \subset \{0,1,\ldots,n\}} \FT/\FTI. \]
Then $\Delta$ is the fundamental domain of the action of $\tilde{\W}(\Gr)$ on $\Sigma$, (or that of $\T \ltimes \LbG$ on $\mathcal{A}$). It is homeomorphic to a simplex with faces indexed by proper subsets of $ \{0,1,\ldots,n\}$. Under the identification of $\Sigma$ with $\lt$, the subspace of $\lt$ corresponding to $\Delta$ is called the Affine Alcove and is defined as:
\[ \Delta = \{ h \in \lt \, | \, \alpha_i(h) \geq 0, \, \, \,  \alpha_0(h) \leq 1 \quad 1 \leq i \leq n \}. \]
The space $\Delta$ may also be seen as an affine subspace of the Lie algebra of $\T \ltimes \LG$ induced by the inclusion $\Tr \subseteq  \LG$. In this identification, the groups $\T \ltimes \Hg_I$ are exactly the stabilizers (under the adjoint action of $\T \ltimes \LG$) of the walls $\Delta_I$ in $\Delta$ corresponding to the subset $I$. Using the homotopy decomposition of $\mathcal{A}$, one obtains an equivariant affine inclusion of $\mathcal{A}$ into the Lie algebra of $\T \ltimes \LG$. We may extend this inclusion to an equivariant inclusion of $i \R_+ \times \mathcal{A}$, where $i \R$ denotes the Lie algebra of the rotation circle $\T$. 

\end{remark}

\section{The equivariant sheaf ${}^k \mathcal{K}_{\T \ltimes \LbG}$ over $\mathfrak{h} \times \Sigma_\C$, and locality:} \label{locality}

\medskip
\noindent
In this section, we will construct a holomorphic sheaf  built from dominant $\K$-groups over the complexification $\mathfrak{h} \times \Sigma_\C$ (to be defined below) of the space $i \R_+  \times \Sigma$. The choice of this space from the standpoint of field theory if explained in detail in remark \ref{complexification} below. 

\medskip
\noindent
We construct this sheaf in two steps: first we construct a local coefficient system ${}^k \mathcal{K}^\ast_{\bf T}$ on $\mathfrak{h} \times \Sigma_\C$, and next we extend ${}^k \mathcal{K}_{\bf T}$ to a module ${}^k \mathcal{K}^\ast_{\T \ltimes \LbG}$ over the sheaf of holomorphic functions on $\mathfrak{h} \times \Sigma_\C$. 

\medskip
\noindent
Let us consider the stack $(i \R_+ \times \mathcal{A}) \qu (\T \ltimes \LbG)$. The coarse moduli space (or orbit space) of this stack is the space $i \R_+ \times \Delta$ described above in remark \ref{affine}. Let $\pi : i \R_+ \times \mathcal{A} \longrightarrow i \R_+ \times \Delta$ denote the projection map. Given a $\T \ltimes \LG$-space ${\mbox Y}$, one obtains a coefficient system ${}^k \mathcal{B}^\ast_{\bf T}(\mbox{Y})$ over $i \R_+ \times \Delta$ given by sheafifying the pre-sheaf:
\[ U \longmapsto {}^k \K^\ast_{\T \ltimes \LbG}(\pi^{-1}(U) \times \mbox{Y}). \]
Let $ (\tau, x) \in i\R_+ \times \Delta$ be a point, with $x$ being in the interior of the wall $\Delta_I \subset \Delta$, recall that the stabilizer of $(\tau, x)$ under the action of $\T \ltimes \LbG$ is the group $\T \ltimes \Hbg_I$. It follows form the definitions that the stalk at $(\tau, x)$ is the $\K$-theory group: ${}^k \K^\ast_{\T \ltimes \Hbg_I}(X)$. Since we are working in characteristic zero, this stalk is canonically isomorphic to the Weyl invariants: ${}^k \K_{{\bf T}}(\mbox{Y})^{\W_I}$. It follows from this description that the above coefficient system is the $\tilde{\W}(\Gr)$-invariants of the push-forward along $i \R_+ \times \Sigma \longrightarrow i \R_+ \times \Delta$ of the $\tilde{\W}(\Gr)$-equivariant constant coefficient system over $i \R_+ \times \Sigma$ with constant value ${}^k \K^\ast_{{\bf T}}(X)$ at each point. We call this constant coefficient system ${}^k \mathcal{K}^\ast_{\bf T}(\mbox{Y})$ over $i \R_+ \times \Sigma$. In the sequel, we shall find it more convenient to work with ${}^k \mathcal{K}^\ast_{\bf T}(\mbox{Y})$ instead of ${}^k \mathcal{B}^\ast_{\bf T}(\mbox{Y})$. It is straightforward to extend ${}^k \mathcal{K}_{\bf T}(\mbox{Y})$ to a coefficient system over the complexification of $i \R_+ \times \Sigma$ as follows. Let us identify the complexification of the positive energy axis $i \R_+$ with the upper half plane $\mathfrak{h}$. Recall that $\Sigma$ was homeomorphic to the Lie algebra $\lt$ of the maximal torus $\Tr$. Hence we may define $\Sigma_\C$ to be $\lt \otimes \C$ (see remark \ref{complexification} below). 

\medskip
\noindent
The space $\mathfrak{h} \times \Sigma_\C$ supports a free affine action of a group $\mathcal{N} = (\pi_1(\Tr) \oplus \pi_1(\Tr)) \rtimes \W(\Gr)$, with the action of $\W(\Gr)$ acting diagonally on both lattices. In particular, we have a canonical map $\mathcal{N} \longrightarrow \tilde{\W}(\Gr)$, given by:
\[ (\beta_1 \oplus \beta_2) \, w \longmapsto \beta_1\, w, \quad \quad \mbox{where} \quad \quad \beta_i \in \pi_1(\Tr), \quad \mbox{and} \quad w \in \W(\Gr). \]
 The action of $\mathcal{N}$ on $\mathfrak{h} \times \Sigma_\C$ is defined as: 
\[ ((\beta_1 \oplus \beta_2) \, w) \ast (\tau, h) = (\tau, w(h) + \tau \beta_1 + \beta_2). \]

\noindent
Given any $\T \ltimes \LG$-space $\mbox{Y}$, consider the infinite loop space: $\Map^{\bf T}\{ \mbox{Y}, \mathcal{F}_k\}$, endowed with an action of $\mathcal{N}$ that factors through the manifest action of $\tilde{\W}(\Gr)$. We therefore obtain an $\mathcal{N}$-equivariant parametrized spectrum given by the projection onto the first factor: 
\[ (\mathfrak{h} \times \Sigma_\C) \times \Map^{\bf T}\{ \mbox{Y}, \mathcal{F}_k\} \longrightarrow \mathfrak{h} \times \Sigma_\C. \]
We use the same notation ${}^k \mathcal{K}^\ast_{\bf T}(\mbox{Y})$ to denote the $\mathcal{N}$-equivariant sheaf on $\mathfrak{h} \times \Sigma_\C$ given by sheafifying the pre sheaf given by the homotopy groups of the space of sections:
\[ {}^k \mathcal{K}^\ast_{\bf T}(\mbox{Y})_U = \pi_{-\ast} \Map^{\bf T}(U \times \mbox{Y}, \mathcal{F}_k \}.\]
As before, the sheaf ${}^k \mathcal{K}^\ast_{\bf T}(\mbox{Y})$ is nothing more than a $\mathcal{N}$-eqivariant constant coefficient system on $\mathfrak{h} \times \Sigma_\C$, with values: ${}^k \K^\ast_{\bf T}(\mbox{Y})$. 

\bigskip
\begin{remark} \label{complexification}
The ``correct" definition of the complexification of the stack $(i \R_+ \times \mathcal{A}) \qu (\T \ltimes \LbG)$ is understood in the framework of BV-BRST quantization in field theory. To begin with, the complexification of the stack $\mathcal{A} \qu (\T \ltimes \LbG)$ should be understood as the ``phase space" or the tangent bundle of $\mathcal{A} \qu (\T \ltimes \LbG)$. Given any (well behaved) second-order equations of motion defined on the gauge fields restricted to the germ of a cylinder, the phase space parametrizes the (classical) solutions since these solutions are uniquely determined by the Cauchy data, or points in the tangent bundle of $ \mathcal{A}$. This tangent bundle has the form: 
\[ \mbox{T}(\mathcal{A}) = \mathcal{A} \times \Omega^1, \]
where $\Omega^1 := \Omega^1(S^1, \mathfrak{g})$ are the one forms on $S^1$ with values in the Lie algebra $\mathfrak{g}$ of $\Gr$. The space $\Omega^1$ is commonly known as conjugate gauge momenta. 

\medskip
\noindent
The gauge action of $\T \ltimes \LbG$ on $\mathcal{A}$ gives rise to an infinitesimal action of $\mbox{Lie}(\T \ltimes \LbG)$ on $\mbox{T}(\mathcal{A})$. We therefore get an action of $\mathbb{G}$ on $\mbox{T}(\mathcal{A})$, where $\mathbb{G}$ is given by the Gauge group extended by conjugate-symmetries: 
\[ \mathbb{G} = (\T \ltimes \LbG) \ltimes \mbox{Lie}(\T \ltimes \LbG). \] 

\noindent
Let $\mathfrak{h}$ be the upper half plane, which we think of as the complexification of the positive axis $i \R_+$. Notice that $\mathfrak{h}$ determines a family of annuli in $\C$ with parametrized boundary: 
\[ \tau \in \mathfrak{h} \longmapsto \mbox{A}_\tau := \{ z \in \C \, | \quad |q| \leq |z| \leq 1, \quad q = \mbox{exp}(2 \pi i \tau). \} \]
Giving $\mathfrak{h}$ the trivial $\mathbb{G}$-action, the stack $(\mathfrak{h} \times \mbox{T}(\mathcal{A})) \qu \mathbb{G}$ can be understood as the stack parametrizing solutions to second-order equations\footnote{as before, we assume good behaviour of these equations} of motion along the germ of the cylinder about the parametrized inner boundary of $\mbox{A}_\tau$, as we vary $\tau$ in $\mathfrak{h}$.

\medskip
\noindent
Next we use a gauge-fixing procedure to reduce the stack $(\mathfrak{h} \times \mbox{T}(\mathcal{A})) \qu \mathbb{G}$ to the stack $(\mathfrak{h} \times \Sigma_\C) \qu \tilde{\W}(\Gr)$, where we recall that $(\mathfrak{h} \times \Sigma_\C)$ is seen as the canonical subspace of the complexified Lie algebra $(\C \times \Sigma_\C)$ of the maximal torus $\T \times \Tr \subset \T \ltimes \LG$ with the induced Affine Weyl action by $\tilde{\W}(\Gr)$. To see this, we proceed by first extending the inclusion $i \R_+ \times \mathcal{A} \subset \mbox{Lie}(\T \ltimes \LG)$ of remark \ref{affine} to an inclusion:
\[   \rho: \mathfrak{h} \times \mbox{T}(\mathcal{A}) = \mathfrak{h} \times \mathcal{A} \times \Omega^1 \longrightarrow \mbox{Lie}(\T \ltimes \LG) \otimes \C, \quad  \quad (\tau, d + \alpha d\theta, \beta d\theta) \longmapsto  \tau(\partial + \alpha) + \beta, \]
where $i\partial$ represents the generator of the Lie algebra of $\T$. It is easy to check that the action of $\T \ltimes \LbG$ on the left hand side is compatible with the complexification of the Adjoint action on the right hand side. Furthermore, the Adjoint action canonically extends to all of $\mathbb{G}$ making $\rho$ a $\mathbb{G}$-equivariant map over $\mathfrak{h}$.

\medskip
\noindent
Then one uses the above formulas to show that the subspace $(\mathfrak{h} \times \Sigma_\C)$ can be identified with the fixed points of $(\mathfrak{h} \times \mbox{T}(\mathcal{A}))$ under the action of the maximal torus ${\bf T} \subset \T \ltimes \LbG \subset \mathbb{G}$. Furthermore, this identification induces an equivalence of the coarse moduli space $\mathcal{M}$ of the stack $(\mathfrak{h} \times \mbox{T}(\mathcal{A})) \qu \mathbb{G}$ with that of $(\mathfrak{h} \times \Sigma_\C) \qu \tilde{\W}(\Gr)$. Now given a point $y \in \mathcal{M}$, let $\mbox{A}_y$ and $\mbox{B}_y$ denote the (conjugacy classes) of automorphism groups of any object over $y$ in the stack $(\mathfrak{h} \times \mbox{T}(\mathcal{A})) \qu \mathbb{G}$ and $(\mathfrak{h} \times \Sigma_\C) \qu \tilde{\W}(\Gr)$ respectively. The above description of the action implies that the group $\mbox{B}_y$ is the Weyl group of the maximal compact factor in the group $\mbox{A}_y$. 

\medskip
\noindent
As before, the coefficient system on $\mathcal{M}$ constructed using the dominant $\K$-theory evaluated on the stack $(\mathfrak{h} \times \mbox{T}(\mathcal{A})) \qu \mathbb{G}$, agrees with the $\tilde{\W}(\Gr)$-invariants of the push-forward of the sheaf ${}^k \mathcal{K}^\ast_{\bf T}$ over $(\mathfrak{h} \times \Sigma_\C)$. This justifies our choice of using the sheaf ${}^k \mathcal{K}^\ast_{\bf T}$ over $(\mathfrak{h} \times \Sigma_\C) \qu \tilde{\W}(\Gr)$. 
\end{remark}

\medskip
\noindent
The final step in this quantization procedure is to extend ${}^k \mathcal{K}^\ast_{\bf T}(\mbox{Y})$ to a sheaf we denote by ${}^k \mathcal{K}^\ast_{\T \ltimes \LbG}(\mbox{Y})$ which admits a module structure over the sheaf $\mathcal{O}(\mathfrak{h} \times \Sigma_\C)$ of holomorphic functions on $\mathfrak{h} \times \Sigma_\C$. We do this by representing characters of $\T \times \Tr$ as functions on the phase space $\mathfrak{h} \times \Sigma_\C$. The relation between the sheaf ${}^k \mathcal{K}^\ast_{\T \ltimes \LbG}(\mbox{Y})$ and modular functors and genus zero conformal blocks is spelled out in remark \ref{cb}. 

\medskip
\begin{defn} \label{chi}
Let $\T_\C \times \Tr_\C$ denote the complexification of the torus $\T \times \Tr$. Consider the holomorphic exponential map: 
\[ \chi : \mathfrak{h} \times \Sigma_\C \longrightarrow \T_\C \times \Tr_\C, \quad   \chi(\tau, h) = (e^{2 \pi i \tau}, \mbox{exp}( 2 \pi i h)). \]
The map $\chi$ can be made $\mathcal{N}$-equivariant, with the action of $\mathcal{N}$ on $\T_\C \times \Tr_\C$ factoring through the affine action of $\tilde{\W}(\Gr)$. Then the characters induce an injective ring map (also denoted $\chi$):
\[ \chi : \Rep(\T_\C \times \Tr_\C) \longrightarrow \mathcal{O}(\mathfrak{h} \times \Sigma_\C). \]
\end{defn}

\medskip
\noindent
The above map $\chi$ allows us to define: 

\medskip
\noindent
\begin{defn}
We define the $\mathcal{N}$-eqivariant sheaf ${}^k \mathcal{K}^\ast_{\Tr \ltimes \LbG}(\mbox{Y})$ over $\mathcal{O}(\mathfrak{h} \times \Sigma_\C)$ on the level of stalks:
\[ {}^k \mathcal{K}^\ast_{\T \ltimes \LbG}(\mbox{Y})_{(\tau, h)} := {}^k \mathcal{K}^\ast_{\bf T}(\mbox{Y})_{(\tau, h)} \hat{\otimes}_{\chi} \mathcal{O}(\mathfrak{h} \times \Sigma_\C)_{(\tau, h)}, \]
where the completed tensor product is defined as:
\[ {}^k \mathcal{K}_{\bf T}^\ast(\mbox{Y})_{(\tau, h)} \hat{\otimes}_{\chi} \mathcal{O}(\mathfrak{h} \times \Sigma_\C)_{(\tau, h)} := \invlim \{ {}^k \mathcal{K}_{\bf T}^\ast(\mbox{Y}_\alpha)_{(\tau, h)} \otimes_{\chi} \mathcal{O}(\mathfrak{h} \times \Sigma_\C)_{(\tau, h)} \}, \]
the inverse limit being taking over all finite ${\bf T}$-equivariant sub-skeleta $\mbox{Y}_\alpha$ of $\mbox{Y}$. 
\end{defn}

\medskip
\noindent
The above sheaf satisfies a fixed point theorem: 
\medskip

\begin{thm}  \label{fixed}
Given a point $(\tau, h) \in \mathfrak{h} \times \Sigma_\C$. Define elements $h_2, h_1 \in \Sigma$ to be the unique elements so that $h = - \tau h_1 + h_2$. 
Let $\R(h)$ denote the sub torus of $\T \times \Tr$ given by the closure of the one-parameter subgroup $(e^{2 \pi i x}, \mbox{exp}(-2 \pi i x h_1))$, with $x \in \R$. Also define $\Z(h)$ to be the closed subgroup of $\Tr$ generated by the element $\mbox{exp}(2 \pi i h_2)$. Then the inclusion of the fixed points $\mbox{Z}(\tau,h) := \mbox{Y}^{\R(h)} \cap \mbox{Y}^{\Z(h)}  \subseteq \mbox{Y}$ induces an isomorphism: 
\[ {}^k \mathcal{K}_{\T \ltimes \LbG}^\ast(\mbox{Y})_{(\tau, h)} \longrightarrow {}^k \K_{\bf T}(\mbox{Z}(\tau,h)) \hat{\otimes}_{\chi} \mathcal{O}(\mathfrak{h} \times \Sigma_\C)_{(\tau, h)}. \]
\end{thm}

\begin{proof}
Working by induction on the $\T \times \Tr$-cells of $\mbox{Y}$, or by invoking the localization theorem in \cite{S2}, one can show that given a subgroup $S \subset \T \times \Tr$, the the restriction to the fixed points: $\mbox{Y}^S \subseteq \mbox{Y}$ induces an isomorphism in equivariant $\K$-theory once we invert the multiplicative set $\Rep(S)_+$ generated characters of $\Rep(\T \times \Tr)$ of the form $e^{\alpha}-1$ for weights $\alpha$ of $\T \times \Tr$ that restrict nontrivially to $S$. We say that the localization of equivariant $\K$-theory: $\K_{\T \times \Tr}(\mbox{Y})[\Rep(S)_+]^{-1}$ is localized on the fixed subspace $\mbox{Y}^S$. Now, a character $e^\alpha - 1$ is invertible in $\mathcal{O}(\mathfrak{h} \times \Sigma_\C)_{(\tau, h)}$ if and only if we have: 
\[ \alpha \langle \tau, h \rangle \notin \Z, \quad \mbox{or equivalently}: \quad \alpha \langle 1, -h_1 \rangle \neq 0, \quad \mbox{or} \quad \alpha \langle 0, h_2 \rangle \notin \Z. \]
The first condition implies that $e^\alpha - 1$ is nontrivial when restricted to $\R(h)$ and the second condition implies that $e^\alpha - 1$ is nontrivial when restricted to $\Z(h)$. The result follows. 
\end{proof}

\medskip
\noindent
Of particular interest to us is the space $\mbox{Y} = \LM$, where $\M$ is a given $\Gr$-space. In this case, the space $\mbox{Z}(\tau,h)$ in the previous theorem has an appealing interpretation. We leave it to the reader to show:

\bigskip
\begin{corr} \label{walls}
 $\LM^{\R(h)}$ can be identified with the space of periodic loops of the form: 
\[ \LM^{\R(h)} = \{ \gamma \, | \, \gamma(e^{2 \pi ix}) = \mbox{exp}(2 \pi i x h_1) \, m, \, \, m \in \M \}. \]
In particular, $\LM^{\R(h)}$ is abstractly homeomorphic to $\M^{exp(2 \pi i h_1)}$. Similarly, we have the identification: $\LM^{\Z(h)} = \Lo(\M^{exp(2\pi i h_2)})$. Consequently, we have an equality of $\T \times \Tr$-spaces: 
 \[  \LM^{\R(h)} \cap \LM^{\Z(h_2)} =  \{ \gamma \in \LM \, | \, \gamma(e^{2 \pi ix}) = \mbox{exp}(2 \pi i x h_1) \, m, \quad   m \in \M^{\langle exp(2 \pi i h_1), \, exp(2 \pi i h_2) \rangle} \}.\]
Note that this space is abstractly homeomorphic to the finite $\T \times \Tr$-space: $\M^{\langle exp(2 \pi i h_1), exp(2 \pi i h_2) \rangle}$. 
\end{corr}

\medskip
\noindent
This observation, along with the above theorem leads us to:

\medskip
\begin{thm}\label{dom}
Given a finite $\Gr$-space $\M$ (i.e. a space $\M$ equivalent to a finite $\Gr$-CW complex), the sheaf ${}^k \mathcal{K}^\ast_{\T \ltimes \LbG}(\LM)$ is a $\mathcal{N}$-equivariant sheaf of $\mathcal{O}(\mathfrak{h} \times \Sigma_\C)$-modules. Furthermore, it is a cohomological functor in $\M$. In particular, one obtains the Mayer-Vietoris sequence in $\M$ for each stalk of ${}^k \mathcal{K}^\ast_{\T \ltimes \LbG}(\LM)$. 
\end{thm}
\begin{proof}
To show that each stalk is a cohomological functor in $\M$, one simply observes that the finiteness of $\LM^{\R(h)}$ allows us to replace the completed tensor product (in the statement of theorem \ref{fixed}) with the standard tensor product. Now the inclusion of algebraic maps to holomorphic germs:  $\Rep(\T \times \Tr) \longrightarrow \mathcal{O}(\mathfrak{h} \times \Sigma_\C)_{(\tau, h)}$ can be shown to be a flat map. Hence, each stalk is a flat extension of a cohomological functor of $\M$, and is therefore itself a cohomological functor. 
\end{proof}

\medskip
\begin{remark} \label{cb}
Let us offer a geometric description of the sheaf ${}^k \mathcal{K}^\ast_{\Tr \ltimes \LbG}(\mbox{Y})$. Given a point in the phase space $x = (\tau, h) \in \mathfrak{h} \times \Sigma_\C$, one may define two different (albeit isomorphic) level $k$-representations on the same underlying Hilbert space by demanding that they are interpolated by the projective operator $\chi(x)$. This is precisely part of the the data describing the modular functor underlying the category of level $k$ representations  that corresponds to the annulus $\mbox{A}_\tau$ (indexed by $\tau \in \mathfrak{h}$ as described in remark \ref{complexification} above). Indeed, the operator $\chi(x)$ can be interpreted as the correlation function generating the conformal blocks over the space $\mathfrak{h}$ interpreted as a moduli space of annuli. We have therefore constructed a functor from the category of level $k$ representations to itself that fixes each underlying vector space, but conjugates the morphisms by $\chi(x)$. This induces a $\T \ltimes \LbG$-equivariant automorphism, which we denote by $\underline{\chi}(x)$, of the space of Fredholm operators on $\mathcal{H}_k$ induced by conjugation with $\chi(x)$. In this context, the fiber of the sheaf ${}^k \mathcal{K}^\ast_{\Tr \ltimes \LbG}(\mbox{Y})$ constructed above, at the point $x$, is the localization of ${}^k \K^\ast_{\bf T}(\mbox{Y}) \otimes \C$ about the fixed points of the automorphism $\underline{\chi}(x)$. 
\end{remark}

\section{The sheaf ${}^k \mathcal{G}^\ast(\M)$ over the universal elliptic curve, and Theta functions:}\label{grojnowski}

\medskip
\noindent
We begin this section with some important definitions: 
\begin{defn}
Define the $\W(\Gr)$-equivariant universal elliptic curve over $\mathfrak{h}$ as the quotient under the action of $(\pi_1(\Tr) \oplus \pi_1(\Tr)) \subset \mathcal{N}$:
\[ \mathcal{E}_{\Tr} := (\mathfrak{h} \times \Sigma_\C) /(\pi_1(\Tr) \oplus \pi_1(\Tr)). \] 
The holomorphic structure sheaf $\mathcal{O}(\mathcal{E}_{\Tr})$ is defined as: $\{ \zeta_\ast \,\mathcal{O}(\mathfrak{h} \times \Sigma_\C) \}^{(\pi_1(\Tr) \oplus \pi_1(\Tr))}$, where $\zeta$ denotes the projection map from $\mathfrak{h} \times \Sigma_\C$ to $\mathcal{E}_{\Tr}$. Similarly, we define the sheaf ${}^k \mathcal{G}^\ast(\M)$ of $\mathcal{O}(\mathcal{E}_{\Tr})$-modules to be: 
\[ {}^k \mathcal{G}^\ast(\M) := \{ \zeta_\ast \,  {}^k \mathcal{K}^\ast_{\T \ltimes \LbG}(\LM) \}^{(\pi_1(\Tr) \oplus \pi_1(\Tr))}. \] Note that the above sheaves are $\W(\Gr)$-equivariant. 
 \end{defn}

\medskip
\noindent
We may also define untwisted $\mathcal{N}$-equivariant sheaves of algebras: $\mathcal{K}^\ast_{\T \ltimes \LbG}(\LM))$, as well as an untwisted $\W(\Gr)$-eqivivariant sheaf: $\mathcal{G}^\ast(\M)$:

\medskip
\noindent
\begin{defn}
Let $\mathcal{K}^\ast_{\T \ltimes \LbG}(\LM)$ be the $\mathcal{N}$-equivariant sheaf of algebras obtained by extending the constant sheaf with values in (usual) $\T \times \Tr$-equivariant $\K$-theory of $\LM$, 
over $\mathcal{O}(\mathfrak{h} \times \Sigma_\C)$. Similarly, let $\mathcal{G}^\ast(\M)$ denote the $\W(\Gr)$-equivariant sheaf of algebras over the sheaf $\mathcal{O}(\mathcal{E}_{\Tr})$ given by $\{ \zeta_* \mathcal{K}^\ast_{\T \ltimes \LbG}(\LM) \}^{(\pi_1(\Tr) \oplus \pi_1(\Tr))} $. Notice in particular that the sheaves ${}^k \mathcal{K}^\ast_{\T \ltimes \LbG}(\LM)$ and ${}^k\mathcal{G}^\ast(\M)$ are modules over $\mathcal{K}^\ast_{\T \ltimes \LbG}(\LM))$ and $\mathcal{G}^\ast(\M)$ respectively. 
\end{defn}

\medskip
\noindent
It is well known \cite{K1} that the action of $\beta \in \pi_1(\Tr) \subset \tilde{\W}(\Gr)$ on the characters $q, e^\alpha, u$ under the decomposition ${\bf T}_\C = \T_\C \times \Tr_\C \times \, \C^\ast $ is given by the formula: 
\[ \beta \, u = u \, e^{\beta^\ast} \, q^{\frac{1}{2}\langle \beta, \beta \rangle},  \quad \beta \, e^\alpha = e^\alpha \, q^{\alpha(\beta)}, \quad \beta \, q = q. \]
where $\beta^\ast$ is the weight dual to $\beta$ and $\langle \,\,, \,\, \rangle$ denotes a canonical positive-definite quadratic form on $\pi_1(\Tr)$ \cite{K1} induced by the Cartan-Killing form. 

\medskip
\noindent
Now recall that $\chi$ of \ref{chi} represented an $\mathcal{N}$-equivariant map from $\mathfrak{h} \times \Sigma_\C$ to $\T_\C \times \Tr_\C$. The above formulas show that $\chi$ extends to an $\mathcal{N}$-equivariant map (also denoted $\chi$): 
\[ \chi : \mathfrak{h} \times \Sigma_\C \times \C  \longrightarrow \T_\C \times \Tr_\C \, \times \, \C, \quad \chi (\tau, h, z) = (e^{2 \pi i \tau}, \mbox{exp}( 2 \pi i h), z), \]
where the action of $\mathcal{N}$ on $\mathfrak{h} \times \Sigma_\C$ extends to an action on $\mathfrak{h} \times \Sigma_\C \times \C$ given by: 
\[ ((\beta_1 \oplus \beta_2) \,  w) \ast (\tau, h, z) = \big( \tau, \, w(h)+ \tau \beta_1 + \beta_2, \, z \, \mbox{exp}(2 \pi i \langle \beta_1, w(h) \rangle +\pi i \tau \langle \beta_1, \beta_1 \rangle)\big). \]

\bigskip

\begin{defn}
Define the central line bundle $\mathcal{L}$\footnote{the dual of $\mathcal{L}$ is also known as the Looijenga line bundle.} to be the $\mathcal{N}$-equivariant bundle:
\[ \mathcal{L} := \mathfrak{h} \times \Sigma_\C \times \C \longrightarrow \mathfrak{h} \times \Sigma_\C. \]
We will use the same notation to denote the $\W(\Gr)$-equivariant line bundle over $\mathcal{O}(\mathcal{E}_{\Tr})$ given by taking orbits under the $(\pi_1(\Tr) \oplus \pi_1(\Tr))$-action defined above: 
\[ \mathcal{L} :=  (\mathfrak{h} \times \Sigma_\C) \times_{(\pi_1(\Tr) \oplus \pi_1(\Tr))} \C \longrightarrow \mathcal{E}_{\Tr}.  \]
Let us also set the notation $\mathcal{L}^{-k}$ to denote the $k$-fold tensor product of the dual line bundle. 
\end{defn}

\bigskip
\begin{thm}
The $\W(\Gr)$-equivariant sheaf ${}^k\mathcal{G}^\ast(\M)$ over $\mathcal{O}(\mathcal{E}_{\Tr})$ is naturally isomorphic to a rank one locally free $\mathcal{G}^\ast(\M)$-module generated by the line bundle $\mathcal{L}^{-k}$. The corresponding result is also true for the sheaf ${}^k \mathcal{K}^\ast_{\T \ltimes \LbG}(\LM)$. 
\end{thm}
\begin{proof}
Since $u$ is the central character, it is easy to see that in cohomological degree zero, ${}^k \mathcal{G}(\ast)$ is canonically isomorphic to $\mathcal{L}^{-k}$. In particular, we have a canonical inclusion: $\mathcal{L}^{-k} \hookrightarrow {}^k \mathcal{G}^0(\M)$ for any $\M$ induced by the projection map $\M \rightarrow pt$. Extending with the module structure over $\mathcal{G}^\ast(\M)$ gives us the isomorphism we seek. Details are straightforward and are left to the reader. 
\end{proof}

\noindent
It is of interest to explore the structure of the space of invariant sections of ${}^k \mathcal{G}^\ast(\M)$. For the case $\M = pt$, the previous theorem identifies this space of sections with the space of $\W(\Gr)$-invariant sections of $\mathcal{L}^{-k}$. To relate these sections to familiar objects, we need: 

\medskip
\begin{claim}
The space of sections of $\mathcal{L}^{-k}$ can be identified with holomorphic functions $\varphi$ on the space $\mathfrak{h} \times \Sigma_\C \times \C$ with the following transformation property for all $\beta \in \pi_1(\Tr)$: 
\[ \varphi (\tau, h + \beta, z) = \varphi (\tau, h, z),  \quad \quad \varphi( \tau, h, z) =  \mbox{exp}(2 \pi i k  \langle \beta, h \rangle + \pi i k \tau \langle \beta, \beta \rangle) \, \varphi(\tau, h + \tau \beta, z). \]
In addition, the function is homogeneous of degree $k$ in the variable $z$:
\[ \varphi(\tau, h, u z) = u^k \varphi(\tau, h, z). \]
\end{claim}

\begin{proof}
These equalities correspond to the standard identification of sections of lines bundles that are obtained via an associated bundle construction (as in the case of $\mathcal{L}^{-k}$), with functions on the total space of the dual bundle. 
\end{proof}

\noindent
Holomorphic functions that satisfy the conditions above are called theta functions of degree $k$. The following corollary is essentially the content of \cite{K1}(Chap. 13):

\medskip
\begin{corr} \label{rep LG}
In cohomological degree zero, the space of $\W(\Gr)$-invariant global sections of ${}^k \mathcal{G}^0(pt)$ is isomorphic to the vector space generated by the $\W(\Gr)$-invariant theta functions of degree $k$ (with respect to the positive definite quadratic form $\langle \, , \, \rangle $ on $\pi_1(\Tr)$). In particular, this space is finite dimensional and has a basis given by the characters of the irreducible level $k$ representations of $\T \ltimes \LbG$. 
\end{corr}

\medskip
\begin{remark}\label{geometric}
The geometric interpretation of the above corollary is straight forward: Given a level $k$ positive energy irreducible representation $V$ of $\T \ltimes \LbG$, it is well known that the action of the Lie algebra of $\T \ltimes \LbG$ on $V$ can be complexified. Now recall form remark \ref{affine} that $i \R_+ \times \Sigma \times \R$ is a subspace of the Lie algebra of $\T \ltimes \LbG$. Hence, for each point : $(\tau, h, z) \in \mathfrak{h} \times \Sigma_\C \times \C$, one has an operator $\psi(\tau, h, z)$ on $V$, which is homogeneous of degree $k$ in the variable $z$ and factors through the complexified torus $\T_\C \times \Tr_\C \times \, \C^\ast$. The operators $\psi(\tau, h, z)$ preserve each (finite dimensional) $\T$-eigenspace of $V$. It follows that $\psi(\tau, h, z)$ gives rise to a nested family of operators, whose trace converges to the germ of a holomorphic function on $\mathfrak{h} \times \Sigma_\C \times \C$, at the point $(\tau, h, z)$. This is precisely an element of the stalk of the sheaf ${}^k \mathcal{K}_{\T \ltimes \LbG}(pt)_{(\tau, h)}$. In this manner each irreducible positive energy level $k$ representation $V$ gives rise to a section of ${}^k \mathcal{K}_{\T \ltimes \LbG}(pt)$. This construction also allows one to identify the image of ${}^k \mathcal{K}_{\T \ltimes \LbG}(pt)$ inside ${}^k \mathcal{K}_{\T \ltimes \LbG}(\LM)$ for any $\Gr$-space $\M$. 
\end{remark}

\medskip
\begin{remark}\label{main1}
In \cite{G}, Grojnowski constructs a sheaf which is now known as Grojnowski's equivariant elliptic cohomology. What the author finds incredible is that this sheaf is constructed in an a-priori fashion starting with the stalks that are prescribed to be as our corollary \ref{walls} suggests. Grojnowski's construction uses equivariant singular cohomology instead of $\K$-theory, but he points out that the stalks are to be seen under the lens of the Chern character. There have been several later versions of Grojnowski's construction by other authors (see \cite{AB}). It appears very likely that our sheaf ${}^k \mathcal{G}^\ast(\M)$ is closely related to Grojnowski's construction, though we have not explored the details (see section \ref{remarks}). Indeed, remarks in \cite{G} suggest that Grojnowski had something like our framework in mind when constructing his equivariant elliptic cohomology.  As the previous remark suggests, the merit of using a geometric description as we have done, is that it allows one to motivate the constructions (see remarks \ref{complexification},\ref{cb},\ref{geometric}) and use positive energy representations to directly construct elements in equivariant elliptic cohomology. 
 \end{remark}

\section{Level $k$ representations of $\T \ltimes \LbT$ and ${}^k \mathcal{G}(\GT)$:} \label{LGT}

\medskip
\noindent
In this section we explore the structure of the sheaf ${}^k \mathcal{G}(\GT)$. 

\smallskip
\noindent
Let $\mathcal{K}^\ast_{\T \times \Tr}((\T \ltimes \LG)/(\T \times \Tr))$ denote the sheaf of $\mathcal{O}(\mathfrak{h} \times \Sigma_\C)$-modules representing the equivariant $\K$-theory: $\K^\ast_{\T \times \Tr}((\T \ltimes \LG)/(\T \times \Tr))$ as constructed earlier. In other words, we define

\[ \mathcal{K}^\ast_{\T \times \Tr}((\T \ltimes \LG)/(\T \times \Tr))_{(\tau, h)} :=  \K^\ast_{\T \times \Tr}((\T \ltimes \LG)/(\T \times \Tr)) \hat{\otimes}_{\chi} \mathcal{O}(\mathfrak{h} \times \Sigma_\C)_{(\tau, h)}. \]
Notice that the space $(\T \ltimes \LT)/(\T \times \Tr)$ supports an $\T \ltimes \LG$-equivariant right action of $\tilde{\W}(\Gr)$ given by: 
\[  w \ast (g \, \T \times \Tr) = g w \, \T \times \Tr. \]
In particular, the $\mathcal{O}(\mathfrak{h} \times \Sigma_\C)$-module $\mathcal{K}^\ast_{\T \times \Tr}((\T \ltimes \LG)/(\T \times \Tr))$ admits another action of the affine Weyl group: $\tilde{\W}(\Gr)$ which preserves each stalk and commutes with the action of the group $\mathcal{N}$. So as to not confuse this action with the action induced via $\mathcal{N}$, we shall refer to this action as the right action: $\tilde{\W}(\Gr)_r$. 
Let $\mathcal{K}^\ast_{\T \times \Tr}((\T \ltimes \LG)/(\T \times \Tr))^{\pi_1(\Tr)_r}$ denote the $\mathcal{N}$-equivariant sheaf of invariants under the subgroup $\pi_1(\Tr) \subseteq \tilde{\W}(\Gr)_r$. Notice that $\mathcal{K}^\ast_{\T \times \Tr}((\T \ltimes \LG)/(\T \times \Tr))^{\pi_1(\Tr)_r}$ admits a residual action of $\W(\Gr)_r$ that commutes with all the present structure. With this observation in place, we have:

\bigskip
\begin{thm} \label{GT}
The $\mathcal{N} \times \W(\Gr)_r$-equivariant sheaf ${}^k \mathcal{K}^\ast_{\T \ltimes \LbG}(\Lgt)$ may be described as: 
\[  {}^k \mathcal{K}^\ast_{\T \ltimes \LbG}(\Lgt) =  \mathcal{K}^\ast_{\T \times \Tr}((\T \ltimes \LG)/(\T \times \Tr))^{\pi_1(\Tr)_r} \otimes \mathcal{L}^{-k} = {}^k \mathcal{K}^\ast_{\T \times \Tr}((\T \ltimes \LG)/(\T \times \Tr))^{\pi_1(\Tr)_r}. \] 
Furthermore, the $\W(\Gr)_r$-invariant sub-sheaf of ${}^k \mathcal{K}^\ast_{\T \ltimes \LbG}(\Lgt)$ can be identified with the image of ${}^k \mathcal{K}^\ast_{\T \ltimes \LbG}(pt)$
induced by the projection map from $\GT$ to a point.  
\end{thm}

\begin{proof}
Using the results and notation of corollary \ref{walls}, we recall that for a point $(\tau, h)$, the stalk ${}^k \mathcal{K}^\ast_{\T \ltimes \LbG}(\Lgt)_{(\tau, h)}$ is localized on a subspace of $(\GT)^{exp(2 \pi i h_1)}$. Now, using the action of $\mathcal{N}$ (which preserves the isomorphism type of the stalks), we may assume that $-h_1$ belongs to the affine alcove $\Delta$. Assume that $-h_1$ belongs to the interior of a wall $\Delta_I$. It follows that $(\GT)^{exp(2\pi i h_1)} = \Hg_I \times_{\No_I} \W$, where we use $\tilde{\W}$, $\W$ and $\W_I$ to denote $\tilde{\W}(\Gr)$, $\W(\Gr)$ and $\W_I(\Gr)$ resp., and $\No_I$ denotes the normalizer of the maximal torus of $\Hg_I$. In particular, we have: 
\[ {}^k \mathcal{K}^\ast_{\T \ltimes \LbG}(\Lgt)_{(\tau, h)} = \K^\ast_{\T \times \Tr}((\T \ltimes \Hg_I) \times_{\No_I} \W) \otimes_{\chi} \mathcal{L}^{-k}_{(\tau, h)},  \]
where the action of $\T$ on $\Hg_I$ is via the map $t \mapsto \mbox{exp}(2 \pi i t h_1)$. However, it is easy to see that one may choose the standard action of $\T$ on $\Hg_I$ without changing the equivariant $\K$-theory (since any two actions differ by a map from $\T$ to the center of $\Hg_I$). Now we have the equality: 
\[ \K^\ast_{\T \times \Tr}((\T \ltimes \Hg_I) \times_{\No_I} \W) = \K^\ast_{\T \times \Tr}((\T \ltimes \Hg_I) \times_{\No_I} \tilde{\W})^{\pi_1(\Tr)_r}. \]
In addition, the space $(\T \ltimes \Hg_I) \times_{\No_I} \tilde{\W}$ is easily seen to be identical to the fixed point space: $(\T \ltimes \LG/(\T \times \Tr))^{exp(2\pi i h_1)}$. Unraveling this sequence of equalities gives rise to a natural isomorphism: 
\[  {}^k \mathcal{K}^\ast_{\T \ltimes \LbG}(\Lgt)_{(\tau, h)} \cong  \mathcal{K}^\ast_{\T \times \Tr}((\T \ltimes \LG)/(\T \times \Tr))^{\pi_1(\Tr)_r} \otimes \mathcal{L}^{-k}_{(\tau, h)}. \]
It remains to explore the $\W_r$-invariant sub-sheaf. For this, consider an arbitrary proper subset $I \subset \{0,1,\ldots,n\}$, and let $\iota_I$ denote the inclusion of fixed points: 
\[ \iota_I : \tilde{\W} \subseteq (\T \ltimes \Hg_I) \times_{\No_I} \tilde{\W}. \]
The maps $\iota_I$ are compatible in $I$ and hence by \cite{HHH} \cite{KK}, give rise to an injection of $\mathcal{N} \times \W_r$-equivariant sheaves: 
\[  \iota^\ast : \mathcal{K}^\ast_{\T \times \Tr}((\T \ltimes \LG)/(\T \times \Tr))^{\pi_1(\Tr)_r} \otimes \mathcal{L}^{-k} \longrightarrow \{ \, \prod_{w \in \tilde{\W}} \mathcal{L}^{-k} \,  \}^{\pi_1(\Tr)_r} =  \prod_{w \in \W} \mathcal{L}^{-k}, \]
where the co-domain is endowed with the action of $\W_r$ induced by right permutations of the indexing set for the product. Taking invariants with respect to $\W_r$, leads easily to the proof of the statement we required. 
\end{proof}

\medskip
\noindent
Next we consider $\mathcal{N}$-invariant global sections of the sheaf ${}^k \mathcal{K}^\ast_{\T \times \Tr}((\T \ltimes \LG)/(\T \times \Tr))^{\pi_1(\Tr)_r}$. This space will be expresed in terms of certain classes that can formally be expressed in terms of Euler classes of certain equivariant line bundles on $(\T \ltimes \LG)/(\T \times \Tr)$. In addition, they are expressible in terms of theta functions. For this reason, we call these classes Euler-Theta classes. For the sake of brevity, we will stick with the notation $\W$ and $\tilde{\W}$ for $\W(\Gr)$ and $\tilde{\W}(\Gr)$ resp. We recall our convention to use the notation $(w \ast e^\lambda)$ to denote the action of $w \in \tilde{\W}$ on the character $e^\lambda$ of $\T \times \Tr \times \, S^1$ via the action of $\mathcal{N}$ (which, we recall, acts on characters along its projection onto $\tilde{\W}$). 

\medskip
\noindent
Now given a character $e^{\lambda}$ of $\T \times \Tr \times \, S^1$ of level $k$, consider the formal theta character: 
\[ \theta_\lambda = \sum_{\beta \in \pi_1(\Tr)} e^{\beta \ast \lambda}. \]
By \cite{K1} (Ch. 12), $\theta_\lambda$ can be seen to be a holomorphic section of the line bundle $\mathcal{L}^{-k}$. By construction, $\theta_\lambda$ is invariant under the action of $(\pi_1(\Tr) \oplus \pi_1(\Tr)) \subset \mathcal{N}$. This leads us to: 

\begin{defn}
Given a level $k$ character $\lambda$ of $\T \times \Tr \times \, S^1$, we define the Euler-Theta class $e(\lambda)$: 
\[ e(\lambda) \in \{ \prod_{w \in \tilde{\W}}  \mathcal{O}(\mathcal{L}^{-k}) \, \nu_w \}^{\pi_1(\Tr)_r}, \quad \quad e(\lambda) = \prod_{w \in \tilde{\W}} (w \ast \theta_\lambda) \,  \nu_w, \]
where $\nu_w$ is a place holder for the factor corresponding to the element $w \in \tilde{\W}$. 
\end{defn} 

\bigskip
\noindent
Consider the injection of $\mathcal{N}$-invariant global sections induced by the inclusion of $\T \times \Tr$-fixed points: 
\[ \Gamma_{\mathcal{N}} \iota^\ast  : \Gamma_{\mathcal{N}} \, {}^k \mathcal{K}^\ast_{\T \times \Tr}((\T \ltimes \LG)/(\T \times \Tr))^{\pi_1(\Tr)_r} \longrightarrow  \{ \prod_{w \in \tilde{\W}} \mathcal{O}(\mathcal{L}^{-k}) \, \nu_w \}^{\pi_1(\Tr)_r} = \prod_{w \in \W} \mathcal{O}(\mathcal{L}^{-k}). \]

\begin{thm} \label{rep LT}
The image of $\Gamma_{\mathcal{N}} \iota^\ast$ is a finite dimensional vector space spanned by $e(\lambda)$, where $\lambda$ ranges over the equivalence class of characters of level $k$ under the action of $\pi_1(\Tr)$ induced via $\mathcal{N}$. In particular, by \cite{K1}, the space of $\mathcal{N}$-invariant global sections of ${}^k \mathcal{K}^\ast_{\T \times \Tr}((\T \ltimes \LG)/(\T \times \Tr))^{\pi_1(\Tr)_r}$ is isomorphic to the vector space spanned by all level $k$ representations of $\T \ltimes \LbT$, under the induced central extension $\T \ltimes \LbT \subset \T \ltimes \LbG$.
\end{thm}

\begin{proof}
Given any level $k$ character $e^\lambda$, consider the expression of the form $\prod_{w \in \tilde{\W}} (w \ast e^{\lambda}) \, \nu_w$. It is easy to see that this is the image (under $\iota^\ast$) of the $\T \ltimes \LbG$-equivariant line bundle over $(\T \ltimes \LG)/(\T \times \Tr)$, induced by the weight $\lambda$. In particular, this element represents an $\mathcal{N}$-invariant global section of the bundle ${}^k \mathcal{K}^0_{\T \ltimes \LbG}((\T \ltimes \LG)/(\T \times \Tr))$. Taking the orbit of this element under the action of $\pi_1(\Tr)_r$ results in an $\mathcal{N} \times \pi_1(\Tr)_r$-invariant element whose factors are expressible as theta functions and therefore $e(\lambda)$ is a well defined $\mathcal{N}$-invariant element in the co-domain of the map $\Gamma_{\mathcal{N}} \iota^\ast$.  It remains to check that $e(\lambda)$ is in the image of $\Gamma_{\mathcal{N}} \iota^\ast$. For this we invoke results of \cite{HHH} that identify the image of $\Gamma_{\mathcal{N}} \iota^\ast$ as follows. Given a positive real root $\alpha$ of $\T \ltimes \LG$, let $r_\alpha \in \tilde{\W}$ denote the reflection corresponding to $\alpha$. Let $v,w$ be elements of $\tilde{\W}$ with the property that $w = r_\alpha v$ is a reduced expression. Then by \cite{HHH}, one needs to verify that $(w \ast \theta_\lambda) - (v \ast \theta_\lambda)$ is divisible by $e^{\alpha}-1$ in $\mathcal{O}(\mathcal{L}^{-k})$. To establish this fact, notice that for any $\beta \in \pi_1(\Tr)$, we have: 
\[ w \ast (\beta \ast e^\lambda) - v \ast (\beta \ast e^\lambda) = e^{v \ast (\lambda + k \beta^\ast)} q^{\lambda(\beta) + \frac{k}{2}\langle \beta, \beta \rangle}\{ e^{-v \ast (\lambda + k \beta^\ast)(h_\alpha) \, \alpha} - 1 \}. \]
For fixed $w, v, \lambda, \alpha$, the above expression may be factored in $\mathcal{O}(\mathcal{L}^{-k})$: 
\[ w \ast (\beta \ast e^\lambda) - v \ast (\beta \ast e^\lambda) = (e^{\alpha}-1) \, \varphi(\beta). \]
Notice that the elements $\varphi(\beta)$ are characters dominated by $q^{\frac{k}{2} \langle \beta, \beta \rangle}$ and therefore the sum over all $\beta$ converges to give us a well defined factorization in $\mathcal{O}(\mathcal{L}^{-k})$: 
\[  (w \ast \theta_\lambda) - (v \ast \theta_\lambda) = (e^{ \alpha}-1) \sum_{\beta \in \pi_1(\Tr)} \varphi(\beta). \]
This proves that the elements $e(\lambda)$ are in the image of $\Gamma_{\mathcal{N}} \iota^\ast$. Now let $n \in \mathcal{N}$ be an arbitrary element. The action of this element on a section $\psi$ of the form $\prod_{w \in \tilde{\W}} \psi_w \, \nu_w$ is given by: 
\[ n \ast (\prod_{w \in \tilde{\W}} \psi_w \, \nu_w) = \prod_{w \in \tilde{\W}} (n \ast \psi_w) \, \nu_{n \ast w}. \]
Hence a $\mathcal{N}$-invariant global section $\psi$ is determined by its factor $\psi_e$ corresponding to the unit $e \in \tilde{\W}$. If in addition, $\psi$ is $\pi_1(\Tr)_r$-invariant, we deduce that $\psi_e$ must be fixed by the action of the lattice $(\pi_1(\Tr) \oplus \pi_1(\Tr)) \subset \mathcal{N}$. Hence, $\psi_e$ is an arbitrary holomorphic section of the line bundle $\mathcal{L}^{-k}$ over $\mathcal{E}_{\Tr}$. These theta functions are known to be a vector space on a basis given by elements $\theta_\lambda$ with $\lambda$ ranging on the quotient space mentioned above. In addition, these theta functions index level $k$ representations of $\T \ltimes \LbT$ \cite{K1} \cite{PS}. 
\end{proof}

\medskip
\begin{remark}
We may construct elements in ${}^k \mathcal{K}^0_{\T \ltimes \LbT}(pt)$ along the lines described in remark \ref{geometric}. These may be induced up to elements in ${}^k \mathcal{K}^0_{\T \ltimes \LbG}(\LGT)$ using the fact that $\LGT = (\LbG)/(\LbT)$. The above theorem implies that this procedure of induction exhausts all elements of the space of global sections of ${}^k \mathcal{K}^\ast_{\T \ltimes \LbG}(\LGT)$.
\end{remark}

\bigskip
\noindent
We end this section with the example of a free $\Gr$-space $\M$: 

\medskip
\begin{thm} \label{G}
Let $\M$ be a free finite dimensional $\Gr$-space, and let $\K_{\T \times \Tr}^\ast(\M) = \K^\ast(\MT)[q^{\pm}]$ denote the $\W(\Gr)$-module given by the topological $\K$-theory of the homogeneous space $\MT$ with a trivial action of $\T$. Let $\mathcal{K}^\ast_{\T}(\MT)$ be the corresponding sheaf on $\T_\C \times \Tr_\C$. Then as $\mathcal{N}$-equivariant sheaves, we have an equivalence: 
\[ {}^k \mathcal{K}^\ast_{\T \ltimes \LbG}(\LM) =  \mathcal{K}_{\T}^\ast(\MT) \otimes_{\chi} \mathcal{L}^{-k}. \] 
Furthermore, ${}^k \mathcal{G}^\ast(\M)$ is supported along the zero section: $\mathfrak{h} \longrightarrow \mathcal{E}_{\Tr}$\footnote{These sheaves appear to not necessarily be induced from $\mathcal{O}(\mathfrak{h})$-modules}.
\end{thm}
\begin{proof}
Consider the stalk at a point $(\tau, h)$. Recall that we showed in corollary \ref{walls} that this stalk is localized on the space of $\M^{exp(2 \pi i h_1)} \cap \M^{exp(2\pi i h_2)}$. But since $\M$ is a free $\Gr$-space, the only way one may get a non-trivial stalk is if $h_1, h_2 \in \pi_1(\Tr)$. By using the action of the lattice in $\mathcal{N}$, we may assume that $h_1 = h_2 = 0$. It follows that ${}^k \mathcal{K}^\ast_{\T \ltimes \LbG}(\LM)$ is supported on the $\mathcal{N}$-orbit of $\mathfrak{h}$, and that it is given by constant loops if $h_1=h_2=0$:
\[ {}^k \mathcal{K}^\ast_{\T \ltimes \LbG}(\LM)_{(\tau, 0)} = \K^\ast_{\T \times \Tr}(\M) \otimes_{\chi} \mathcal{L}^{-k}_{(\tau, 0)} = \K^\ast(\MT)[q^{\pm}] \otimes_{\chi} \mathcal{L}^{-k}_{(\tau, 0)}. \]
\end{proof}

\section{Modularity of ${}^k \mathcal{G}(\M)$:} \label{modularity}

 \medskip
 \noindent
 Grojnowski has pointed out in \cite{G} that his sheaf is modular, i.e. it is equivariant under an action of $\mbox{SL}_2(\Z)$ that extends the action on $\mathcal{E}_{\Tr}$. This issue of modularity is somewhat subtle in our case. It will turn out that the untwisted sheaves $\mathcal{G}(\M)$ admit an action of $\mbox{SL}_2(\Z)$ compatible with the action of $\mathcal{N}$ whenever the $\Gr$-space $\M$ has a certain property which we shall make precise in the next paragraph. However, an interesting double cover of the group generated by $\mbox{SL}_2(\Z)$ and $\mathcal{N}$ will act on the line bundles $\mathcal{L}^{-k}$. By tensoring these actions together, we obtain an action of this double cover on the sheaf ${}^k \mathcal{G}(\M)$. 
 
 \medskip
 \noindent
 Let us fix a category of $\Gr$-spaces $\M$ that will be of interest to us. The property we will assume on our spaces $\M$ is familiar in the literature. It was first studied by Goresky, Kottwitz and MacPherson in \cite{GKM} and later explored by several authors. The context relevant to us has been studied in \cite{HHH} where the equivariant $\K$-theory of $\M$ is described combinatorially. We shall call these spaces GKM-spaces: 
 
 \begin{defn} \cite{HHH}
 Given a $\Gr$-space $\M$, consider $\M$ as a $\Tr$-space by restriction. We call $\M$ a GKM-space if $\M$ admits a $\Tr$-equivariant stratification: $\M = \coprod_{I} \mathcal{U}_I$, with a single $\Tr$-fixed point in the stratum $\mathcal{U}_I$ denoted by $F_I$. We make three assumptions on this stratification: 
 
 \begin{itemize} 
 \item We assume that the space obtained by collapsing the lower strata from the closure of a stratum: $\overline{\mathcal{U}}_I/\coprod_{J < I} \mathcal{U}_J$ is the compactfication of a $\Tr$-representation $V_I$ about the fixed point $F_I$.
 
 \item
  We assume that $V_I$ can be decomposed as: 
 \[ V_I = \bigoplus_{J < I} V_{IJ}, \]
 where $V_{IJ}$ is a sub representation such that its unit sphere maps to the fixed point $F_J$ under the attaching map to the lower strata. 
 
 \item
 Finally, we also assume that the $\Tr$-equivariant $\K$-theoretic Euler classes of $V_{IJ}$ are all mutually relatively prime. 
 \end{itemize}
 \end{defn}
  
  \medskip
  \noindent
  \begin{example}
 Given a parabolic subgroup $\Hg_I \subseteq \Gr$ for $I \subseteq \{1, \ldots, n\}$, the homogeneous space $\GHgI$ is an example of a GKM-space. 
  
  \end{example}
 \medskip
 \noindent
 The main theorem in \cite{HHH} states:
 
 \medskip
 \noindent
 \begin{thm} \cite{HHH} \label{HHH}
Given a GKM-space $\M$, the restriction map to the fixed points: 
 \[ \K^\ast_{\Tr}(\M) \longrightarrow \prod_{I} \Rep(\Tr), \]
 is injective, with the image given by elements $\prod_I \alpha_I$ so that $\alpha_I - \alpha_J$ is divisible by the Euler class of $V_{IJ}$ for all $J \leq I$. 
 \end{thm}
 
 \medskip
 \noindent
 We will use the above theorem to show that the sheaf $\mathcal{K}^\ast_{\T \ltimes \LbG}(\LM)$ is modular in a sense to be made precise below. 
 
 \medskip
 \noindent
\begin{defn} 
Consider the (right) action of $\mbox{SL}_2(\Z)$ on $(\pi_1(\Tr) \oplus \pi_1(\Tr))$ commuting with $\W(\Gr)$: 
\[ \left( \begin{array}{cc} a & b \\ c & d \end{array} \right ) \ast (\beta_1 \oplus \beta_2) = (a \beta_1 + c \beta_2, \, b \beta_1 + d \beta_2). \]
We may therefore define an extension of $\mathcal{N}$ by $\mbox{SL}_2(\Z)$ which we denote $\mathcal{N}_2(\Z)$: 
\[ \mathcal{N}_2(\Z) = (\pi_1(\Tr) \oplus \pi_1(\Tr)) \rtimes (\W(\Gr) \times \mbox{SL}_2(\Z)). \]
More precisely, the new relations in $\mathcal{N}_2(\Z)$ are of the form: 
\[ \left( \begin{array}{cc} a & b \\ c & d \end{array} \right )^{-1} (\beta_1 \oplus \beta_2) \left( \begin{array}{cc} a & b \\ c & d \end{array} \right ) =  (a \beta_1 + c \beta_2, \, b \beta_1 + d \beta_2). \]
We may extend the action of $\mathcal{N}$ on $\mathfrak{h} \times \Sigma_\C$ to an action of $\mathcal{N}_2(\Z)$ by defining a left $\mbox{SL}_2(\Z)$ action: 
\[  \left( \begin{array}{cc} a & b \\ c & d \end{array} \right ) \ast (\tau, h) = \left(  \frac{ a \tau + b}{c \tau + d} \, ,  \frac{h}{c \tau + d}  \right).  \]
\end{defn}

\medskip
\noindent
\begin{claim}
Given a GKM-space $\M$, the action of $\mathcal{N}_2(\Z)$ on $\mathfrak{h} \times \Sigma_\C$ induces an action on the stalks of the sheaf $\mathcal{K}^\ast_{\T \ltimes \LbG}(\LM)$. In particular, the stalks of the untwisted sheaf $\mathcal{G}(\M)$ admit an action of the group $\W(\Gr) \times \mbox{SL}_2(\Z)$. 
\end{claim}
\begin{proof}
Recall from \ref{walls} that the stalks $\mathcal{K}_{\T \ltimes \LbG}(\LM)_{(\tau, h)}$ are localized on the $\T \times \Tr$-space: 
 \[ \mbox{Z}(\tau, h) :=  \{ \gamma \in \LM \, | \, \gamma(e^{2 \pi ix}) = \mbox{exp}(2 \pi i x h_1) \, m, \quad   m \in \M^{\langle exp(2 \pi i h_1), \, exp(2 \pi i h_2) \rangle} \}. \]
 The action of the group $\T \times \Tr$ on $\mbox{Z}(\tau,h)$ factors through the map: 
 \[ \rho(\tau,h) : \T \times \Tr \longrightarrow \Tr, \quad \quad (e^{2 \pi i x},s) \longmapsto e^{2\pi i (h_1 x-h_2)} s \]
 where we recall from \ref{fixed} that $h_1, h_2$ are defined uniquely by the equation: $h = -h_1 \tau + h_2$. Now the stalks of $\mathcal{K}^\ast_{\T \ltimes \LbG}(\LM)$ can be described as: 
\[ \mathcal{K}^\ast_{\T \ltimes \LbG}(\LM)_{(\tau, h)} =  \K^\ast_{\Tr}(\mbox{Z}(\tau, h))\otimes_{\chi_{(\tau,h)}} \mathcal{O}(\mathfrak{h} \times \Sigma_\C)_{(\tau, h)},\]
where the map $\chi_{(\tau,h)}$ is is induced by the composite of $\chi$ with the map $\rho(\tau,h)$ above: 
\[ \chi_{(\tau,h)} : \mathfrak{h} \times \Sigma_\C \longrightarrow \Tr_\C, \quad \quad (z,s) \longmapsto \mbox{exp}(2\pi i (h_z + s)), \]
where $h_z = zh_1-h_2$. Notice that the map $\chi_{(\tau,h)}$ maps the element $(\tau,h)$ to the trivial element in $\T_\C$. 

\medskip
\noindent
The action of a matrix in $A \in \mbox{SL}_2(\Z)$ on $\mathfrak{h} \times \Sigma_\C$ sends the pair $(h_1, h_2)$ to the pair $(\hat{h}_1, \hat{h}_2)$, where the pairs are related by: 
\[ a \hat{h}_1 - c \hat{h}_2 = h_1, \quad d \hat{h}_2 -b \hat{h}_1  = h_2, \quad A =  \left( \begin{array}{cc} a & b \\ c & d \end{array} \right ). \]
Observe that the group generated by $\mbox{exp}(2 \pi i h_0)$ and $\mbox{exp}(2 \pi i h_1)$ remains unchanged under the action of $A$. Hence $\mbox{Z}(A(\tau,h))$ is canonically equivalent to $\mbox{Z}(\tau,h)$ as a $\Tr$-space. This defines the operator induced by the $\mbox{SL}_2(\Z)$-action on $\mathfrak{h} \times \Sigma_\C$:
\[ A^\ast : \mathcal{O}(\mathfrak{h} \times \Sigma_\C)_{A(\tau,h)} \longrightarrow \K^\ast_{\Tr}(\mbox{Z}(\tau, h))\otimes_{A^\ast \chi_{A(\tau,h)}} \mathcal{O}(\mathfrak{h} \times \Sigma_\C)_{(\tau, h)}, \quad A^\ast \psi(z,s) = \psi(A(z,h)). \]
Next, we show that the right hand side is canonically isomorphic to $\mathcal{O}(\mathfrak{h} \times \Sigma_\C)_{(\tau,h)}$. For this we invoke theorem \ref{HHH}. Since $\M$ is a GKM-space, it follows that $\mbox{Z}(\tau,h)$ is also a GKM-space with the same set of fixed points. Consider the injective restriction map: 
\[  \K^\ast_{\Tr}(\mbox{Z}(A(\tau, h)))\otimes_{A^\ast \chi_{A(\tau,h)}} \mathcal{O}(\mathfrak{h} \times \Sigma_\C)_{(\tau, h)} \longrightarrow \prod_I \Rep(\Tr) \otimes _{A^\ast \chi_{A(\tau,h)}} \mathcal{O}(\mathfrak{h} \times \Sigma_\C)_{(\tau, h)} = \prod_I \mathcal{O}(\mathfrak{h} \times \Sigma_\C)_{(\tau, h)}. \]
Notice that a similar map also exists for  $\K^\ast_{\Tr}(\mbox{Z}(\tau, h))\otimes_{\chi_{(\tau,h)}} \mathcal{O}(\mathfrak{h} \times \Sigma_\C)_{(\tau, h)} $. We proceed to show that the ring: $\K^\ast_{\Tr}(\mbox{Z}(A(\tau, h)))\otimes_{A^\ast \chi_{A(\tau,h)}} \mathcal{O}(\mathfrak{h} \times \Sigma_\C)_{(\tau, h)}$ is canonically isomorphic to the ring $\K^\ast_{\Tr}(\mbox{Z}(\tau, h))\otimes_{\chi_{(\tau,h)}} \mathcal{O}(\mathfrak{h} \times \Sigma_\C)_{(\tau, h)}$, with both objects seen inside $\prod_I \mathcal{O}(\mathfrak{h} \times \Sigma_\C)_{(\tau, h)}$. Consider an element:
 \[ \alpha := \prod_I \alpha_I  \in \prod_I \mathcal{O}(\mathfrak{h} \times \Sigma_\C)_{(\tau, h)}, \quad A^\ast \chi_{A(\tau,h)}^\ast e(V_{IJ}) | (\alpha_I - \alpha_I), \]
 where $e(V_{IJ})$ denotes the equivariant $K$-theoretic Euler class of $V_{IJ}$. By theorem \ref{HHH}, $\alpha$ is precisely an element in the image of $\K^\ast_{\Tr}(\mbox{Z}(A(\tau, h)))\otimes_{A^\ast \chi_{A(\tau,h)}} \mathcal{O}(\mathfrak{h} \times \Sigma_\C)_{(\tau, h)}$. To show that $\alpha$ is also in the image of $\K^\ast_{\Tr}(\mbox{Z}(\tau, h))\otimes_{\chi_{(\tau,h)}} \mathcal{O}(\mathfrak{h} \times \Sigma_\C)_{(\tau, h)}$, it is sufficient to show that $A^\ast \chi_{A(\tau,h)}^\ast e(V_{IJ})$ is a unit times $\chi_{(\tau,h)}^\ast e(V_{IJ})$. For this, recall that $\chi_{(\tau,h)}$ sends the point $(\tau,h)$ to the unit in $\Tr_\C$. It follows from this that the virtual characters $A^\ast \chi_{A(\tau,h)}^\ast e(V_{IJ})$ and $\chi_{(\tau,h)}^\ast e(V_{IJ})$ have a  zero of the same order (given by the dimension of $V_{IJ}$) at the point $(\tau,h)$. Hence the ratio of $A^\ast \chi_{A(\tau,h)}^\ast e(V_{IJ})$ and $\chi_{(\tau,h)}^\ast e(V_{IJ})$ is well-defined, and represents an invertible element in $\mathcal{O}(\mathfrak{h} \times \Sigma_\C)_{(\tau, h)}$, which is what we wanted to show. 
 
 \medskip
 \noindent
 The above identification yields a canonical map: 
\[ A^\ast :  \mathcal{O}(\mathfrak{h} \times \Sigma_\C)_{A(\tau,h)} \longrightarrow  \K^\ast_{\Tr}(\mbox{Z}(\tau, h))\otimes_{\chi_{(\tau,h)}} \mathcal{O}(\mathfrak{h} \times \Sigma_\C)_{(\tau, h)} = \mathcal{O}(\mathfrak{h} \times \Sigma_\C)_{(\tau,h)} . \]
We leave it to the reader to check that this map of stalks extends to an action of $\mathcal{N}_2(\Z)$: 
\[ \mathcal{K}^\ast_{\T \ltimes \LbG}(\LM)_{A(\tau, h)} \longrightarrow  \mathcal{K}^\ast_{\T \ltimes \LbG}(\LM)_{(\tau, h)}. \]
\end{proof}

\bigskip
\noindent
It remains to demonstrate modularity of ${}^k \mathcal{G}(\M)$. For this, we need to establish a modular action on the line bundles $\mathcal{L}^{-k}$. As mentioned earlier, this turns out to be a subtle matter. We begin by constructing a double cover of $\mathcal{N}_2(\Z)$ using a cocycle $\eta$ defined below. Recall the positive definite form $\langle \,, \rangle$ on $\pi_1(\Tr)$ induced by the Cartan-Killing form. Consider the $\W(\Gr)$-invariant quadratic form mod two: 
\[ \mu : \pi_1(\Tr) \oplus \pi_1(\Tr) \longrightarrow \Z/2, \quad \quad (\alpha, \beta) \longmapsto \langle \alpha, \beta \rangle \mod 2. \]
Given $A \in \mbox{SL}_2(\Z)$, and $(\beta_1 \oplus \beta_2) \in \pi_1(\Tr) \oplus \pi_1(\Tr)$, we also define the $\Z/2$ valued function:
\[ \eta(\beta_1 \oplus \beta_2, A) := \mu(A \ast (\beta_1 \oplus \beta_2))-\mu(\beta_1 \oplus \beta_2). \]

\bigskip
\begin{remark}
One may check that $\eta$ is trivial if $\pi_1(\Tr)$ is an even lattice i.e. $\langle \beta, \beta \rangle \in 2 \Z$ for all $\beta \in \pi_1(\Tr)$. This is the case if $\Gr$ is simply laced. 
\end{remark}

\medskip
\noindent
\begin{defn}
Define the double cover $\mathcal{M}_2(\Z)$ of $\mathcal{N}_2(\Z)$ by extending the action of $\mbox{SL}_2(\Z)$ on the lattice $(\pi_1(\Tr) \oplus \pi_1(\Tr))$ using $\eta$ as a cocycle: 
\[ \mathcal{M}_2(\Z) = (\Z/2 \, \oplus \, \pi_1(\Tr) \oplus \pi_1(\Tr)) \rtimes (\W(\Gr) \times \mbox{SL}_2(\Z)), \]
with $\Z/2$ being central, and relations: 
 \[ A^{-1} (\beta_1 \oplus \beta_2) \, A =  \eta(\beta_1 \oplus \beta_2, A) \, \oplus \, A \ast (\beta_1 \oplus \beta_2). \]
 Note that this central extension is canonically split over $\mathcal{N}$ and $\mbox{SL}_2(\Z)$.
 \end{defn}
 
 \medskip
 \noindent
We may now extend the action of $\mathcal{N}$ on the line bundle $\mathcal{L}$, to an action of $\mathcal{M}_2(\Z)$ that lifts the action of $\mathcal{N}_2(\Z)$ on $\mathfrak{h} \times \Sigma_\C$ as follows:  Given $A \in \mbox{SL}_2(\Z)$, we define a left action of $A$ on the line bundle $\mathcal{L} = \mathfrak{h} \times \Sigma_\C \times \C$ by:
\[  \left( \begin{array}{cc} a & b \\ c & d \end{array} \right ) \ast (\tau, h, z) = (\frac{a \tau + b}{c \tau + d} \, , \frac{h}{c \tau + d}, z \, \mbox{exp}(-\frac{\pi i \, c }{c \tau + d} \langle h, h \rangle),\]
where $\langle h, h \rangle$ denotes the $\C$-linear extension of the quadratic form on $\pi_1(\Tr)$. The generator of the central $\Z/2$ is defined to act by multiplication with $-1$ on the factor $\C$. We leave it to the reader to check that this defines an action of $\mathcal{M}_2(\Z)$ on $\mathcal{L}$.

\medskip
\noindent
As an immediate consequence of these observations, we have: 

\medskip
\noindent
\begin{thm} \label{modular}
Given a GKM-space $\M$, the global sections of the sheaf ${}^k \mathcal{K}^\ast_{\T \ltimes \LbG}(\LM)$ admit an action of the group $\mathcal{M}_2(\Z)$. In the case of an even lattice $\pi_1(\Tr)$, or even level $k$, global sections of ${}^k \mathcal{G}(\M)$ admit an action of the group $\W(\Gr) \times \mbox{SL}_2(\Z)$. 
\end{thm}

\section{Some comments on our construction:} \label{remarks}

 \medskip
 \noindent
In the previous section, we demonstrated modularity for the sheaf ${}^k \mathcal{G}(\M)$ in the case of GKM-spaces $\M$. The proof of this fact used the restriction to $\Tr$-fixed points of $\M$. That proof does not extend to arbitrary $\Gr$-spaces $\M$ in a straightforward manner, which is in contrast to the construction made by Grojnowski, where no special property is assumed for $\M$. This may suggest that either our construction is different from that of Grojnowski, or perhaps that a different proof is needed to show the equivalence of the two constructions. One may speculate that the failure of modularity for arbitrary $\M$ (if that is indeed the case) could be seen as a statement that there is no viable CFT that couples a sigma model on $\M$ with a rational CFT with the required gauge symmetries. The question of modularity remains open. 

\medskip
\noindent
Our sheaves ${}^k \mathcal{G}(\M)$ are $\Z/2$-graded since stalks are constructed from dominant $\K$-theory. In particular, setting $\M$ to be a point, the sheaf ${}^k \mathcal{G}(pt)$ is simply the vector-space of holomorphic sections of the line bundle $\mathcal{L}^{-k}$ graded in even parity. It is natural to extend the grading of the sheaf ${}^k \mathcal{G}(\M)$ to the integers by defining the sheaf in degree $\ast + 2n$ to be the sheaf ${}^k \mathcal{G}^\ast(\M)$ twisted with $\omega ^{\otimes n}$, where $\omega$ denotes the pullback to $\mathcal{E}_{\Tr}$ of the $\mbox{SL}_2(\Z)$-invariant line bundle of ``invariant differentials" on $\mathfrak{h}$. Notice that after we grade our sheaves over the integers, the $\mathcal{M}_2(\Z)$-invariant global sections of ${}^k \mathcal{G}^\ast(pt)$ can be identified with $\W$-invariant Jacobi forms of several variables.

\bigskip

\pagestyle{empty}
\bibliographystyle{amsplain}

\begin{thebibliography}{10}
\bibitem[AB]{AB} M. Ando, M. Basterra, \textit{The Witten genus and equivariant elliptic cohomology}, Mathmatische Zeitschrift, No. 240, (2002), 787--822. 
\bibitem[FHT]{FHT} D. Freed, M. Hopkins, C. Teleman, \textit{Loop groups and twisted K-theory, III}, Ann. Math., 174, (2011), 947--1007. 
\bibitem[G]{G} I. Grojnowski, \textit{Delocalised Equivariant Elliptic cohomology}, Elliptic Cohomology, London Math. Soc. Lecture Note Series (342), Cambridge Univ. Press, (2007), 114--121. 
\bibitem[GKM]{GKM} M. Goresky, R. Kottwitz, R. MacPherson, \textit{Equivariant cohomology, Koszul duality, and the localization theorem}, Invent. Math. 131, (1998), 25--83. 
\bibitem[HHH]{HHH} M. Harada, A. Henriques, T. Holm, \textit{Computation of generalized equivariant cohomologies of Kac-Moody flag varieties}, Advances in Math., 197, No. 1, (2005), 198--221. 
\bibitem[K1]{K1} V. Kac, \textit{Infinite dimensional Lie algebras}, Cambridge University Press, 1990. 
\bibitem[Ki]{Ki} N. Kitchloo, \textit{Dominant K-theory and Integrable highest weight representation of Kac-Moody groups}, Advances in Math., 221, (2009), 1227-1246. 
\bibitem[KK]{KK} B. Kostant, S. Kumar, \textit{T-equivariant $\K$-theory of generalized flag varieties}, J. Diff. Geom., 32, 1990, 549--603. 
\bibitem[KM]{KM} N. Kitchloo, J. Morava, \textit{Thom prospectra and Loop group representations}, Elliptic Cohomology, London Math. Soc. Lecture Note Series (342), Cambridge Univ. Press, (2007), 214-238. 
\bibitem[L]{L} J. Lurie, \textit{A survey of Elliptic Cohomology}, available at: http://www.math.harvard.edu/~lurie/. 
\bibitem[NC]{NC} \textit{Interpretation of Quantum Field Theory/String theory}, available at: http://ncatlab.org/nlab/show/equivariant+elliptic+cohomology\#InterpretationInQuantumFieldTheory.
\bibitem[PS]{PS} A. Presley, G. Segal, \textit{Loop Groups}, Oxford University Press, 1986. 
\bibitem[S]{S} G. Segal, \textit{The definition of a conformal field theory}, Topology, Geometry and Quantum Field theory, London Math. Soc. LEcture Note Ser., 308, Cambridge Univ. Press, (2004), 421--577. 
\bibitem[S2]{S2} G. Segal, \textit{Equivariant K-theory}, Pub. Math. de. I.H.E.S., 34 (1968), 129--151. 
\bibitem[S3]{S3} G. Segal, \textit{Elliptic cohomology}, Sem. Bourbaki, 1987-88, No. 695, 187--201. 
\end{thebibliography}
\providecommand{\bysame}{\leavevmode\hbox
to3em{\hrulefill}\thinspace}

\end{document}